\documentclass[11pt,a4paper]{article}

\usepackage[margin=2.5cm]{geometry}

\usepackage[latin1]{inputenc}
\usepackage{amsmath,amsfonts,amssymb,amsthm}

\usepackage{bbm}
\usepackage{enumitem}

\usepackage[colorlinks=true,citecolor=black,linkcolor=black,urlcolor=black]{hyperref}
\usepackage{hyphenat}
\usepackage{appendix}

\usepackage{chngcntr}

\usepackage[numbers]{natbib}
\usepackage[font=small,labelfont=bf]{caption}

\usepackage{wrapfig}

\newtheorem{theorem}{Theorem}
\newtheorem*{theorem*}{Theorem}
\newtheorem{lemma}[theorem]{Lemma}

\newtheorem*{claim*}{Claim}

\newcommand{\E}{\mathbb{E}}
\newcommand{\Pb}{\mathbb{P}}

\newcommand{\Gnp}{G(n,p)}

\newcommand{\mc}{\mathcal}

 \title{Random triangles in random graphs}

\renewcommand{\le}{\leqslant}
\renewcommand{\ge}{\geqslant}
\renewcommand{\epsilon}{\varepsilon}

\author{Annika Heckel
\thanks{Mathematical Institute, University of Oxford,
Andrew Wiles Building, Woodstock Road, Oxford OX2~6GG, UK. E-mail: 
\texttt{heckel@maths.ox.ac.uk}. 
Research supported by ERC Grant 676632.
}
}

\begin{document}
\maketitle

\begin{abstract}
In \cite{riordan2018random}, Oliver Riordan shows that for $r \ge 4$ and $p$ up to and slightly larger than the threshold for a $K_r$-factor, the hypergraph formed by the copies of $K_r$ in $G(n,p)$ contains a copy of the binomial random hypergraph $H=H_r(n,\pi)$ with $\pi \sim p^{r \choose 2}$. 
For $r=3$, he gives a slightly weaker result where the density in the random hypergraph is reduced by a constant factor. Recently, Jeff Kahn announced an asymptotically sharp bound for the threshold in Shamir's hypergraph matching problem for all $r \ge 3$. With Riordan's result, this immediately implies an asymptotically sharp bound for the threshold of a $K_r$-factor in $\Gnp$ for $r \ge 4$. In this note, we resolve the missing case $r=3$ by modifying the argument in \cite{riordan2018random}. This means that Kahn's result also implies a sharp bound for triangle factors in~$\Gnp$.
\end{abstract}

\section{Introduction}
For $r \ge 2$, $n \ge 1$, $\pi=\pi(n) \in [0,1]$, we denote by $H_r(n,\pi)$ the binomial random $r$-uniform hypergraph where each of the ${n \choose r}$ potential hyperedges is included independently with probability~$\pi$. In \cite{riordan2018random}, Oliver Riordan showed that for $r \ge 4$ and $p$ up to and slightly beyond $n^{-2/r}$, the hypergraph formed by the copies of $K_r$ in the random graph $\Gnp=H_2(n,p)$ contains a copy of $H_r(n,\pi)$ with almost the same density.
\begin{theorem}[\cite{riordan2018random}]\label{theoremcoupling}
Let $r \ge 4$ be given. There exists some $\epsilon = \epsilon(r)>0$ such that, for any $p=p(n) \le n^{-2/r + \epsilon}$, the following holds. For some $\pi=\pi(n) \sim p^{r \choose 2}$, we may couple the random graph $G=G(n,p)$ with the random hypergraph $H=H_r(n,\pi)$ so that, whp\footnote{We say that an event $E=E(n)$ holds \emph{with high probability} (whp) if $\lim_{n \rightarrow \infty} \Pb(E) =1$.}, for every hyperedge in $H$ there is a copy of $K_r$ in $G$ with the same vertex set.
\end{theorem}
In particular, Theorem \ref{theoremcoupling} applies when $p$ is in the range of the threshold of a $K_r$-factor in $\Gnp$, or accordingly when $\pi$ is in the range of the threshold for a complete matching in $H_r(n,\pi)$, both of which were famously determined up to a constant factor by Johansson, Kahn and Vu \cite{jkv}. Recently, Jeff Kahn announced a proof that the threshold for a complete matching in $H_r(n,\pi)$ is at $\pi\sim (r-1)!n^{-r+1} \log n$, giving an asymptotically sharp answer to Shamir's problem. Together with Theorem \ref{theoremcoupling}, this immediately carries over to $K_r$-factors in $\Gnp$, implying a sharp threshold at $p \sim \left((r-1)! \log n\right)^{1/{r \choose 2}}n^{-2/r} $ for $r \ge 4$.

For $r=3$, the proof in \cite{riordan2018random} only gives a weaker result where $\pi$ is a constant fraction of~$p^3$. In this note, we show that Theorem \ref{theoremcoupling} also holds for $r=3$, modifying the proof in~\cite{riordan2018random}. This means that Kahn's result also implies a sharp threshold for a triangle factor in $G(n,p)$ at $p \sim \left(2 \log n\right)^{1/3} n^{-2/3}$.

\begin{theorem}
 The conclusion of Theorem \ref{theoremcoupling} also holds for $r=3$.
\end{theorem}

\section{Proof}
The original proof fails for $r=3$ because of the presence of certain problematic configurations in $H$, the \emph{clean $3$-cycles}. These consist of three hyperedges where each pair meets in exactly one distinct vertex. Let $\Gamma$ denote the set of all potential clean $3$-cycles, then we say $\gamma \in \Gamma$ is in $H$ if the corresponding hyperedges are present. In a slight abuse of notation, we will also call an edge configuration where each such hyperedge is replaced by a triangle a clean $3$-cycle, and we say that $\gamma\in \Gamma$ is in $G$ if the corresponding edges are present.

 Our strategy is to first choose which clean $3$-cycles are present in $G$ and $H$, coupling their distributions so that whp we pick the same $3$-cycles for both $G$ and $H$. Conditioning on the event that $G$ and $H$ contain exactly these clean $3$-cycles, we run a modified version of the coupling argument from \cite{riordan2018random} where the bad case can no longer happen. For the sake of brevity, we do not repeat the entire argument from \cite{riordan2018random} but only describe the modifications.

As in the original proof, we will show that if our coupling fails, then either the maximum degree of the final hypergraph $H$ is too high or $H$ contains a certain type of sub-hypergraph called an `avoidable configuration', both of which only happens with probability $o(1)$. Define the avoidable configurations as in Definition 7 of the original proof, then by Lemma 8 in \cite{riordan2018random}, whp $H$ contains no avoidable configurations as long as we pick $\pi \le n^{-2+\epsilon'}$ for some small $\epsilon'>0$.

In \cite{riordan2018random}, the proof of Lemma 9 only fails for $r=3$ in one particular case, namely if vertices of the $K_3$ in question form the \emph{middle triangle} of a clean $3$-cycle in $H$. By this we mean the three vertices in which the hyperedges of the clean $3$-cycle meet (which is not a hyperedge in the clean $3$-cycle itself). Therefore, for $r=3$ the proof gives the following variant of Lemma 9.
\begin{lemma}\label{lemma6}
Let $H$ be a $3$-uniform hypergraph, and let $G$ be the simple graph obtained by replacing each hyperedge of $H$ by a triangle. If $G$ contains a triangle $T$ and the corresponding hyperedge is not present in $H$, then either the vertices of $T$ are the middle triangle of a clean $3$-cycle in $H$, or $H$ contains an avoidable configuration.
\end{lemma}
Let $X_1$ and $X_2$ denote the numbers of clean $3$-cycles in $G=G(n,p)$ and in $H=H_3(n,\pi)$, and let $\lambda_1= 120{n \choose 6}p^9=\E X_1$ and $\lambda_2 =120{n \choose 6}\pi^3=\E X_2$. If $p \le n^{-2/3+\epsilon}$, then $\lambda_1 = O(n^{9\epsilon})$. As in the proof of Theorem~1, we can later pick $\pi=(1-n^{-\delta})p^3$ for some constant $\delta>0$ (see Remark 2 in \cite{riordan2018random}). Decreasing $\epsilon$ if necessary, we can therefore assume $\lambda_1=\lambda_2-o(1)$. Let $\mc{C}_1$ and $\mc{C}_2$ be the collection of all clean $3$-cycles in $G$ and in $H$, respectively.
\begin{lemma}
$\mc{C}_1$ and $\mc{C}_2$ can be coupled so that whp $\mc{C}_1=\mc{C}_2$.
\end{lemma}
\begin{proof}
For two random variables $W,Z$ taking values in a countable set $\Omega$, let
\[
 d_{\text{TV}}(W,Z) = \frac{1}{2}\sum_{\omega \in \Omega}\left| \Pb\left( W =\omega\right) - \Pb\left( Z=\omega\right)\right|
\]
denote their \emph{total variation distance}.  From Theorem 4.7 in \cite{ross2011fundamentals} (which originally appeared in \cite{arratia1989two}),
 the total variation distance between the distributions of $X_1$ and $X_2$ and the Poisson distributions $\mathbf{Poi}(\lambda_1)$ and $\mathbf{Poi}(\lambda_2)$ is $o(1)$, respectively. As $\lambda_2=\lambda_1-o(1)$, the total variation distance between $\mathbf{Poi}(\lambda_1)$ and $\mathbf{Poi}(\lambda_2)$ is also $o(1)$, and so $ d_{\text{TV}}(X_1,X_2)=o(1)$.

 Both in $G$ and in $H$, whp all clean $3$-cycles are pairwise vertex disjoint since $\lambda_1, \lambda_2=O(n^{9 \epsilon})$ (decreasing $\epsilon$ if necessary). Let $i \in \{1,2\}$. Denote by $\Gamma'$ the set of all collections of clean $3$-cycles which are not pairwise vertex disjoint, then $\Pb(\mc{C}_i \in \Gamma')=o(1)$. 
 For $t\ge 0$, let $\Gamma_t$ be the set of all collections of $t$ disjoint clean $3$-cycles. Conditional on $X_i=t$ and $\mc{C}_i \notin \Gamma'$, by symmetry $\mc{C}_i$ is uniformly distributed on $\Gamma_t$. Therefore,
 \begin{align*}
  d_{\text{TV}} (\mc{C}_1, \mc{C}_2) &\le \frac{1}{2}\sum_t \sum_{\gamma \in \Gamma_t} \left|\frac{\Pb(X_1=t)}{|\Gamma_t|}-\frac{\Pb(X_2=t)}{|\Gamma_t|}\right| +\Pb(\mc{C}_1 \in \Gamma')+\Pb(\mc{C}_2\in \Gamma')\\
  &=d_{\text{TV}}(X_1,X_2)+o(1)=o(1).
 \end{align*}
Since the total variation distance of $\mc{C}_1$ and $\mc{C}_2$ is $o(1)$, their distributions can be coupled so that whp $\mc{C}_1=\mc{C}_2$.
\end{proof}

We start the construction of $G=G(n,p)$ and $H=H_3(n,\pi)$ by choosing $\mc{C}_1$ and $\mc{C}_2$, coupling their distributions so that whp $\mc{C}_1=\mc{C}_2$. If $\mc{C}_1 \neq \mc{C}_2$, we say that the coupling has failed.  We assume that the clean $3$-cycles in $\mc{C}_1=\mc{C}_2$ are pairwise vertex disjoint, which holds with probability $1-o(1)$, otherwise we also say the coupling has failed. Let $C_1$ be the set of edges and $C_2$ be the set of hyperedges in the revealed clean $3$-cycles. Let $\mathcal{L}_1$ and $\mathcal{L}_2$ be the events that $G$ and $H$ contain no other clean $3$-cycles, respectively.

We now proceed with the coupling as in Algorithm 11 in \cite{riordan2018random}, revealing the hyperedges of $H$ and some triangles of $G$ one by one (skipping those which we already included with the clean $3$-cycles). At step $j$, we calculate the conditional probability $\pi_j$ of the triangle edge set $E_j$ being present in $G(n,p)$ {and} the \emph{conditional probability} $\pi'_j$ of the corresponding hyperedge $h_j$ being present in $H_3(n,\pi)$, based on the information revealed so far, the edges and hyperedges in $C_1$ and $C_2$, and the events $\mathcal{L}_1$ and $\mathcal{L}_2$. As in \cite{riordan2018random}, as long as $\pi_j' \le \pi_j$ we are ok: we flip a coin with success probability $\pi'_j / \pi_j$, and in the case of success test for the triangle in $G$, including the edge $h_j$ in $H$ iff the coin succeeds and the triangle was included in $G$. If $\pi_j' > \pi_j$, we include $h_j$ in $H$ with probability $\pi'_j$, and if this happens the coupling fails. After we have done this for every hyperedge, $H$ is constructed with the correct distribution, and we pick $G$ with the conditional distribution of $G(n,p)$ given the revealed information. It remains to show that for an appropriate choice of $\pi=p^3(1-o(1))$, the probability that the coupling fails is $o(1)$.

As in \cite{riordan2018random}, we assume for notational simplicity that $p\le n^{-2/3+o(1)}$, although it is clear from the proof that the argument goes through if $p\le n^{-2/3+\epsilon}$ for some small constant $\epsilon>0$. As in \cite{riordan2018random}, there is some $\Delta=n^{o(1)}$ so that whp, every vertex in $H_3(n,\pi)$ has degree at most $\Delta/3$. Let $\mc{B}_1$ denote the bad event that some vertex in the final version of $H$ has degree more than $\Delta/3$, so $\Pb(\mc{B}_1)=o(1)$. Let $\mc{B}_2$ be the event that the final version of $H$ contains an avoidable configuration, then $\Pb(\mc{B}_2)=o(1)$. We will see that if our coupling fails, then $\mc{B}_1 \cup \mc{B}_2$ holds. Let $A_i$ denote the event that the triangle $E_i$ is in $G$.

Suppose we have reached step $j$ of the algorithm where we test for the hyperedge $h_j$ and the event $A_j$. First note that we always have $\pi'_j \le \pi$. To see this, consider the random hypergraph $H'$ where all the revealed hyperedges and the hyperedges from $C_2$ are included, and all hyperedges we have found not to be present so far are excluded, and all other hyperedges are present independently with probability $\pi$. Then $\mathcal{L}_2$ is a down set in the product probability space corresponding to $H'$, and the event that the hyperedge $h_j$ is present is an up set, so
\[
\pi_j'= \Pb(h_j \in H' \mid \mathcal{L}_2) \le \Pb(h_j \in H') =\pi.
\]
Even though this is not how we started the coupling, we can think of the state of $G$ and $H$ at step $j$ as though we had started by testing for all clean $3$-cycles $\gamma \in \Gamma$ in $G$ and in $H$, and received the answer `yes' for $\gamma \in \mc{C}_1$ and the answer `no' for all other $\gamma \in \Gamma$. Then similarly as in \cite{riordan2018random}, let $ R$ be the set of edges found to be in $G$ so far (both from the revealed triangles in the first $j-1$ steps \emph{and} from $C_1$). Let $N$ denote the set of all $i < j$ where we tested for $A_i$ and received the answer `no', and also add an index $i$ to $N$ for every $\gamma \in \Gamma \setminus \mc{C}_1$ (i.e., we add an element to $N$ for every clean $3$-cycle we have excluded). For easier notation, we will now also
 write $E_i$ for the edge set of a clean $3$-cycle with index $i \in N$. Let $N_1$ be the set of all $i\in N$ such that $E_i \cap E_j \neq \emptyset$. Now we can bound $\pi_j$ from below exactly as in equation (4) in \cite{riordan2018random},
\[
 \pi_j \ge p^3 (1-Q) \text{ where } Q=Q_j=\sum_{i \in N_1}p^{|E_i \setminus (E_j \cup R)|}.
\]
It remains to bound $Q$, showing that either $Q=o(1)$, or that if not and the coupling fails, $\mc{B}_1 \cup \mc{B}_2$ holds.

The contribution to $Q$ from all $i$ where $E_i$ is a triangle (rather than a clean $3$-cycle) can be bounded exactly as in \cite{riordan2018random} as long as $\mc{B}_1$ does not hold. Crucially, the previous `bad case' is no longer a problem: suppose that $j$ is `dangerous', i.e.\ there is a triangle $E_i$ with  $i\in N_1$ and $E_i \subset E_j \cup R$. This means that in the previous step $i < j$, we tested for the triangle $E_i$ in $G$ and received the answer `no'. But then $E_i$ cannot be the middle triangle in any clean $3$-cycle in the final version of $H$ --- we know what all the clean $3$-cycles are in both $G$ and  $H$, and if $E_i$ were the middle triangle of one, its edges would have been included in $G$ from the start of the coupling. But then $\pi_i=1$, and if we had tested for $E_i$ we would have received the answer `yes'. So if the coupling fails at step $j$, as $E_i \subset E_j \cup R$,  by Lemma \ref{lemma6} $H$ contains a bad configuration, so $\mc{B}_2$ holds.

Therefore, the contribution to $Q$ from all $E_i$ which are triangles is either $o(1)$, or if not and the coupling fails, $\mc{B}_1 \cup \mc{B}_2$ holds.

Now consider the contribution to $Q$ from some $E_i$, $i \in N_1$, which is a clean $3$-cycle. We want to bound $e_i=|E_i \setminus E(S)|$ from below, where $S$ is the graph on the vertex set of $E_i$ with the edges from $E_j \cup R$ on that vertex set. Suppose $S$ has $k+1$ components, where $0 \le k \le 4$ ($S$ cannot have six components as $E_i \cap E_j$ contains at least one edge). Then $e_i$ is at least the number of edges in $E_i$ between the components of $S$. This can be bounded from below by the minimum number of edges between different parts of a clean $3$-cycle if we partition its vertices into $k+1$ parts --- it is straightforward to check that for $k=1$, $e_i \ge 2$, for $k=2$, $e_i \ge 4$, for $k=3$, $e_i \ge 6$, and for $k=4$, $e_i \ge 8$.

In the connected case where $k=0$, if $e_i=0$, then $E_i \subset E_j \cup R$. Suppose this is the case and the coupling fails, then by Lemma \ref{lemma6}, either the final version of $H$ contains an avoidable configuration and $\mc{B}_2$ holds, or all three triangles $T_1$, $T_2$, $T_3$ of $E_i$ 
 are each either present as hyperedges in $H$ or the middle triangles of a clean $3$-cycle in $H$. Denote the corresponding hyperedges by $t_1$, $t_2$, $t_3$.  
At most one of them can be the middle triangle of a clean $3$-cycle, because we assumed that all clean $3$-cycles are vertex disjoint. Not all $t_i$, $i \in \{1,2,3\}$ are present in $H$ because then the clean $3$-cycle corresponding to $E_i$ would be present, but $i \in N$. So exactly one triangle, say $T_1$, is the middle triangle of a clean $3$-cycle, and $t_2$ and $t_3$ are present in $H$. But then this clean $3$-cycle and $t_2$ and $t_3$ form an avoidable configuration (it can easily be checked that Definition 7 in  \cite{riordan2018random} applies; note that in the hypergraph $H_0$ under consideration, $v(H_0) \le 9$, $e(H_0)=5$, $c(H_0)=1$, so $n(H_0) \ge 2$). Therefore, $\mc{B}_2$ holds.

So if $k=0$, $e_i=0$ and the coupling fails, then $\mc{B}_2$ holds. So suppose $e_i \ge 1$ for all $E_i$ where $k=0$.

As in equation (6) of the original proof, as long as $\mc{B}_1$ does not hold, there are at most $O(n^{k+o(1)})$ instances $i$ where $S$ has $k+1$ components. Therefore, either the contribution to $Q$ from all $E_i$ which are clean $3$-cycles is at most
\[n^{o(1)} \left(p +np^2+n^2p^4+n^3p^6+n^4p^8 \right) =o(1),\]
or if not and the coupling fails, $\mc{B}_1 \cup \mc{B}_2$ holds. Noting that we always have $\pi'_j \le \pi$, we can choose $\pi \sim p^3$ so that whp the coupling does not fail. As in the original proof, it is in fact possible to pick $\pi=p^3(1-n^{-\delta})$ for a small constant $\delta>0$. \qed

  \bibliographystyle{plainnat}

\end{document}